\documentclass[12pt]{article}
\usepackage{amsmath, epsfig, amssymb, amsfonts,color, enumitem, mathtools, amsthm}
\usepackage[left=2cm,top=3cm,right=2cm,bottom=3cm]{geometry}

\usepackage[english]{babel}
\newtheorem{theorem}{Theorem}
\newtheorem{lemma}[theorem]{Lemma}
\newtheorem{definition}[theorem]{Definition}
\newtheorem{corollary}[theorem]{Corollary}

\title{A Note on the Parallel Cleaning of Cliques} 
\author{A. Angeli Ayello\thanks{Department of Mathematics, ETH Zurich, Zurich, Switzerland; email: {\tt angelia@ethz.ch}.},~M.~E.~Messinger\thanks{Department of Mathematics and Computer Science, Mount Allison University, Sackville, NB, Canada; email: {\tt mmessinger@mta.ca}.}}

\begin{document} 
\maketitle

\begin{abstract} We disprove a conjecture proposed in [Gaspers et al., {\it Discrete Applied Mathematics}, 2010] and provide a new upper bound for the minimum number of brushes required to continually parallel clean a clique.  
\end{abstract}


\bigskip


 \section{Introduction and Definitions}\label{sec:intro}

In a graph cleaning model, every vertex and edge of a graph is iniitally considered to be contaminated or {\it dirty} and brushes are distributed to a set of vertices.  A vertex may be cleaned if it contains as many brushes as dirty incident edges.  When a vertex is cleaned, it sends exactly one brush along each dirty incident edge, cleaning those edges.  In the {\it sequential} cleaning model (see~\cite{note, MNP, greedy} for example), at each step exactly one vertex is cleaned and the {\it brush number} of a graph $G$ is defined as the minimum number of brushes needed to clean $G$ using the sequential cleaning model.  In the {\it parallel} cleaning model (see~\cite{parallel, MNP} for example), at each step every vertex that may be cleaned, is cleaned simultaneously and the parallel brush number for a graph $G$ is the minimum number of brushes needed to clean $G$ using the parallel cleaning model.   In~\cite{MNP}, the authors showed that for any graph $G$, the sequential and parallel brush numbers coincide, thus we denote by $b(G)$, the brush number of $G$.

Figures~\ref{figSeq} and~\ref{figPara} illustrate the sequential and parallel cleaning models on a $5$-cycle where one vertex initially has $2$ brushes and all other vertices initially have $0$ brushes.  The dotted lines and white vertices indicate clean edges and clean vertices and for the end of each step, the distribution of brushes (i.e. the number of brushes at each vertex) is given.  The reader will observe that in Figure~\ref{figSeq}, there is a choice as to the second vertex cleaned (and also the third and fourth vertices cleaned).  It was shown in~\cite{MNP} that such decisions can be made arbitrarily as they do not affect whether all vertices (and edges) of a graph can be cleaned; that is, whether a graph can be cleaned depends entirely on the number and initial distribution of brushes.  We further observe that at the end of step $4$, every edge has been cleaned, but one vertex has not yet been cleaned. This example illustrates that after all edges have been cleaned, one additional step may be required to ensure all vertices are clean.  Finally, in the parallel model, it is important to note that if adjacent vertices are cleaned during the same step, both vertices will send a brush along the common edge to the other vertex.  This can be observed in Figure~\ref{figPara}.  

\begin{figure}[htbp]
\[\includegraphics[width=0.8\textwidth]{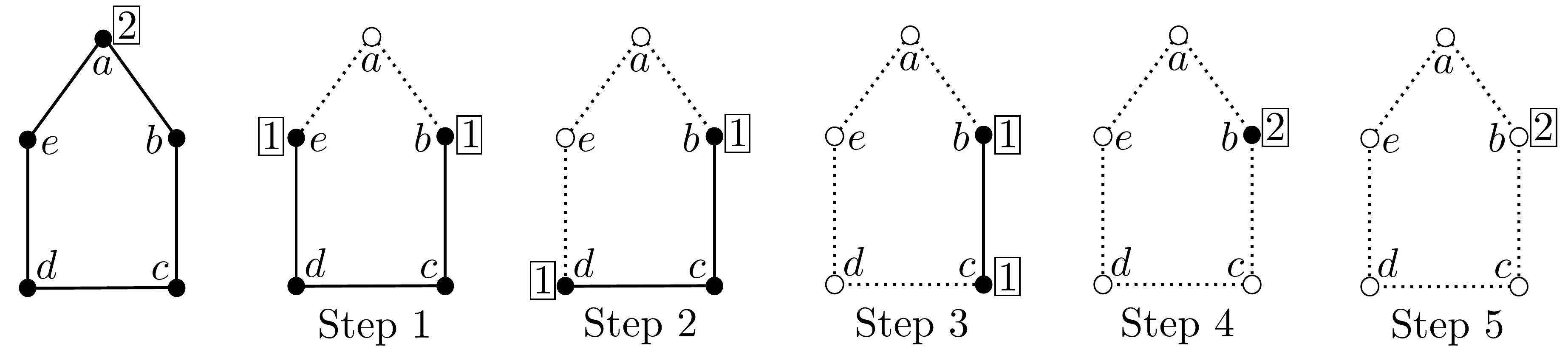}\]
\caption{Sequential cleaning model with $2$ brushes initially on one vertex of $C_5$.}
\label{figSeq}
\end{figure}

\begin{figure}[htbp]
\[\includegraphics[width=0.5\textwidth]{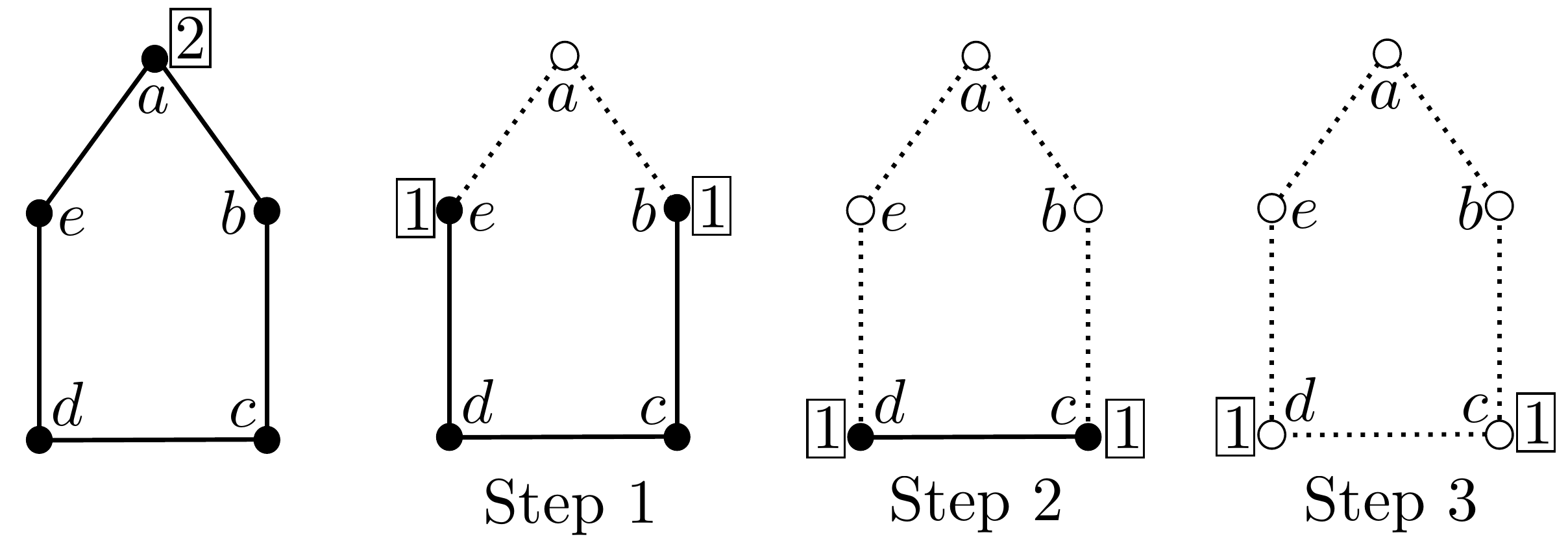}\]
\caption{Parallel cleaning model with $2$ brushes initially on one vertex of $C_5$.}
\label{figPara}
\end{figure}

In this note, we are concerned with the number of brushes required to {\it continually} parallel clean a complete graph (clique).  The sequential cleaning model considers a network that must be cleaned periodically of a regenerating contaminant.  In practice, mechanized brushes are sometimes used to remove regenerating contaminants such as algae and zebra mussels from water pipes as routine maintenance~\cite{Zebra, Zebra2} because zebra mussels can accumulate and restrict water flow in municipal, industrial, and private water systems~\cite{Zebra1}.  As a result, we are interested in whether locations of the brushes after a system has been cleaned, can be used as starting locations for the brushes to clean the system again.  The sequential cleaning model is inherently reversible (see Theorem 2.3 in~\cite{MNP}); that is, a final configuration of brushes on a graph $G$ is always a viable initial configuration of brushes that can be used to clean $G$ again.  Although $b(G)$ brushes can be used to parallel clean a graph $G$ once, the parallel model is not always ``reversible'' (see for example, Figure~\ref{figPara}).  Thus, for many graphs, additional brushes beyond $b(G)$ are required in order to {\it continually} parallel clean the graph and the continual parallel brush number is denoted $cpb(G)$.  

Formally, at each step $t$, $\omega_t(v)$ denotes the number of
brushes at vertex $v$ ($\omega_t: V \rightarrow \mathbb{N} \cup
\{0\}$) and $D_t$ denotes the set of dirty vertices. An edge $uv\in
E$ is dirty if and only if both $u$ and $v$ are dirty: $\{u,v\}
\subseteq D_t$. Finally, let $D_t(v)$ denote the number of dirty
edges incident to $v$ at step $t$:
\begin{equation*}
D_t(v)=\
\begin{cases}
|N(v) \cap D_t| & \textrm{\ if\ \ }  v \in D_t\\
0               & \textrm{\ otherwise.}
\end{cases}
\end{equation*}

We next formally define the parallel graph cleaning process, following the definitions provided in~\cite{parallel}.
\begin{definition}
The \textbf{parallel cleaning process} $\mathfrak{C}(G,\omega_0)=\{(\omega_t,
D_t)\}_{t=0}^K$ of an undirected graph $G=(V,E)$ with an
\textbf{initial configuration of brushes} $\omega_0$ is as follows:

\begin{description}
  \item[(0)] Initially, all vertices are dirty: $D_0=V$;  set $t:=0$

  \item[(1)] Let $\rho_{t+1} \subseteq D_t$
  be the set of vertices such that $\omega_t(v) \geq D_t(v)$ for $v \in \rho_{t+1}$.
  If $\rho_{t+1} = \emptyset$, then stop the process ($K=t$), return the
  \textbf{parallel cleaning sequence}
  $\rho =(\rho_1,\rho_2,\dots,\rho_K)$,
  the \textbf{final set of dirty vertices} $D_K$,
  and the \textbf{final configuration of brushes} $\omega_K$

  \item[(2)] Clean each vertex $v \in \rho_{t+1}$ and all dirty incident edges by
  traversing a brush from $v$ to each dirty neighbour.
  More precisely, $D_{t+1}=D_t \setminus \rho_{t+1}$,
  for every $v \in \rho_{t+1}$, $\omega_{t+1}(v)=\omega_t(v) - D_t(v) + |N(v) \cap
  \rho_{t+1}|$,
  and for every $u \in D_{t+1}$,
  $\omega_{t+1}(u)=\omega_t(u)+|N(u) \cap \rho_{t+1}|$
  the other values of $\omega_{t+1}$ remain the same as in $\omega_t$

  \item[(3)] $t:=t+1$ and go back to {\rm \textbf{(1)}}.
\end{description}
\end{definition}

\begin{definition}
A graph $G=(V,E)$ \textbf{can be cleaned} by the initial
configuration of brushes $\omega_0$ if the cleaning process
$\mathfrak{C}(G,\omega_0)$ returns an empty final set of dirty vertices
($D_T=\emptyset$).\end{definition}

\begin{definition}\label{defn:1}
Let $G$ be a network with initial configuration $\omega_0^0=\omega_0$.
Then $G$ can be \textbf{continually cleaned} using the parallel
cleaning process beginning from configuration $\omega_0$ if for each
$s \in \mathbb{N} \cup \{0\}$, $G$ can be cleaned in parallel using
initial configuration $\omega_0^s$, yielding the final configuration
$\omega_{K_s}^s$ where $\omega_0^{s+1} = \omega_{K_s}^s$.

\smallskip

The \textbf{continual parallel brush number}, $cpb(G)$, of a network
$G$ is the minimum number of brushes needed to continually clean $G$
using a parallel cleaning process.
\end{definition}

In~\cite{parallel}, the authors provided bounds for $cpb$ for a number of graphs and determined $cpb$ exactly for some classes of graphs.  In particular, they showed 
\begin{equation}\label{bound}\frac{5}{16}n^2 + O(n) \leq cpb(K_n) \leq \frac{4}{9}n^2+O(n).
\end{equation}
Based on these bounds and computational results, the authors~\cite{parallel} conjectured 
\begin{equation}\label{conj}\lim_{n\rightarrow \infty} \frac{b(K_n)}{cpb(K_n)} = 9/16.
\end{equation} 
The main result of this note, stated below, provides an improved upper bound for $cpb(K_n)$ which disproves the above conjecture~(\ref{conj}) of~\cite{parallel}.  \\

\noindent{\bf Theorem~\ref{cor:end}.} Let $n_0$ be a non-negative integer and for $i \in \mathbb{Z}^+$, let $n_i = 3n_{i-1}+d_i$ for $d_i \in \{1,2,3\}$.  Then $$cpb(K_{n_i}) \leq  \Big[ \frac{3}{7}+\frac{1}{63}\Big(\frac{2}{9}\Big)^{i+1}\Big] n_i^2 + O(n_i).$$ \medskip

In~\cite{MNP}, it was determined that $b(K_n)=\lfloor \frac{n^2}{4}\rfloor$.  Combined with the results of Theorem~\ref{cor:end}, the following corollary is immediate.\\ 

\noindent{\bf Corollary~\ref{theorem:lim}.} Let $n_0$ be a non-negative integer and for $i \in \mathbb{Z}^+$, let $n_i = 3n_{i-1}+d_i$ for $d_i \in \{1,2,3\}$. Then $$\lim_{i \to \infty}\ \frac{b(K_{n_i})}{cpb(K_{n_i})}\ge\frac{7}{12}.$$  \medskip

\noindent The proofs of Theorem~\ref{cor:end} and Corollary~\ref{theorem:lim} can be found in Section~\ref{sec:results}.  

\section{Results}\label{sec:results}

\begin{definition}  Label the vertices of $K_n$ as $v_0,v_1,\dots,v_{n-1}$ and let $\omega_0$ be an initial configuration of brushes that will parallel clean $K_n$, leaving final configuration $\omega_K$.   Then $\omega_0$ is a  \textbf{1-clique configuration} if 

(1) $\omega_0(v_i) \leq n-1$ for all $i \in \{0,1,\dots,n-1\}$ and 

(2) there is a one-to-one correspondence between the elements of $\{ \omega_0(v_0),\omega_0(v_1),\dots,\omega_0(v_{n-1})\}$ and $\{\omega_K(v_0),\omega_K(v_1),\dots,\omega_K(v_{n-1})\}$.
\end{definition}

For a $1$-clique configuration $\omega_0$, let $S_n(\omega_0) = \sum_{i=0}^{n-1} \omega_0(v_i).$  If $\omega_0$ is an arbitrary initial $1$-clique configuration, we denote $S_n(\omega_0)$ as simply $S_n$.   Then certainly, $cpb(K_n) \leq S_n.$

For $n \equiv 1,2$ (mod $3$), the initial configurations given in~\cite{parallel} that achieve the upper bound of~(\ref{bound}) are $1$-clique configurations, however, the initial configuration given for $n \equiv 0$ (mod $3$) in~\cite{parallel} is not a $1$-clique configuration.  Having $1$-clique configurations are key to our main result later in this section, so in Theorem~\ref{thm:n=3kbetterbound}, we provide a $1$-clique configuration for $K_n$ for $n \equiv 0$ (mod $3$) that uses $\frac{4}{9}n^2+O(n)$ brushes (the proof is similar to that of  Theorem 4.8 in~\cite{parallel}).  

A vertex is said to be {\it primed} if it has at least as many brushes as incident dirty edges.  Vertices are cleaned in three phases:  in phase $1$, a set of $k$ vertices are cleaned,  starting with the only primed vertex, then two primed vertices, then four primed vertices, and so on (although the cardinality of the last subset of vertices need not be a power of $2$).  In phase $2$, a set of $k+3$ vertices are cleaned all in one step.  In phase $3$, the remaining set of $k$ vertices are cleaned in one step, but being a clique, the number of brushes at each vertex does not change during this step. 

\begin{theorem}\label{thm:n=3kbetterbound}
Let $n=3k+3$  for some non-negative integer $k$ and label the vertices of $K_n$ as $v_0,v_1,\dots,v_{3k+2}$. If $$ \omega_0(v_i) = \begin{cases} k+2 & \textrm{ if } i= k,k+1,\dots,2k+2 \\ i & \textrm{ otherwise,}\end{cases}$$ then $K_{3k+3}$ can be cleaned with $1$-clique configuration $\omega_0$, using a total of $4k^2+7k+6 = \frac{4}{9}n^2+O(n)$ brushes. \end{theorem}

\begin{proof}  For $k \in \{0, 1, \dots, 8\}$ we can manually check that $\omega_0$ is a $1$-clique configuration.  Thus, we consider $k >8$.  

Consider the vertices cleaned in phase $1$: $\{v_{3k+2},v_{3k+1},v_{3k},\dots,v_{2k+3}\}$.  Only $v_{3k+2}$ is cleaned during step $1$.  Suppose that $2^{j-1}$ vertices are cleaned during step $j$ for $j \in \{2,3,\dots, t-1\}$ where $t < \lceil \log_2(k+1)\rceil$.  We inductively show that during step $t$, $2^{t-1}$ vertices are cleaned.  Let $v_i$ be a vertex cleaned during step $t$.  Then $$\omega_{t-1}(v_i)  =   i+(2^0+2^1+\cdots+2^{t-2}) =i+2^{t-1}-1 \geq D_{t-1}(v_i) = 3k+2-(2^{t-1}-1) ~\Longrightarrow i \geq 3k+4-2^t.$$  As $v_i$ could not have been cleaned during the previous step, $\omega_{t-2}(v_i) < D_{t-2}(v_i)$, which implies $i< 3k+4-2^{t-1}.$  Thus, during step $t < \lceil \log_2(k+1)\rceil$, $2^{t-1}$ vertices are cleaned.  Finally, we observe that for $v_i$ cleaned during step $t < \lceil \log_2(k+1)\rceil$, $$\omega_t(v_i) = \omega_{t-1}(v_i) - D_{t-1}(v_i)+(2^{t-1}-1) = i+3 \cdot 2^{t-1}-3k-5.$$

We consider the remaining vertices of phase 1; that is, the vertices cleaned during step $\ell = \lceil \log_2(k+1)\rceil$. Let $v_i$ be one such vertex.  Then $\omega_{\ell - 1}(v_i) = i + 2^{\ell - 1} - 1$, and $D_{\ell - 1}(v_i) = 3k + 3 - 2^{\ell -1}$. It follows that $$\omega_{\ell - 1}(v_i) - D_{\ell - 1}(v_i) \geq 2k+ 3 + 2^{\ell - 1} - 1 - (3k + 3 - 2^{\ell -1}) \geq  2^{\ell} - (k + 1) \geq  2^{\ell} - 2^{\log_2(k + 1)} \geq 0$$ since $\ell = \lceil \log_2(k + 1)\rceil \geq  \log_2(k + 1)$. Therefore the remaining vertices of phase 1 are cleaned during step $\ell$.  Since there are a total of $k-(2^{\ell-1}-1)-1$ vertices other than $v_i$ cleaned during step $\ell$, $$\omega_\ell(v_i) = \omega_{\ell-1}(v_i)-(2k+3) = i+2^{\ell-1}-2k-4.$$

We next consider the vertices cleaned during phase 2: $\{v_{2k+2},v_{2k+1},\dots,v_k\}$.  No vertex $v_i \in \{v_{2k+2},v_{2k+1},\dots,v_k\}$ can be cleaned prior to step $\ell+1$ as $$\omega_{\ell-1}(v_i) = \omega_0(v_i)+2^{\ell-1}-1 = k+1+2^{\ell-1} < D_{\ell-1}(v_i) = 3k+3-2^{\ell-1}.$$  However as exactly $k$ vertices were cleaned during phase 1, $$\omega_{\ell}(v_i) = \omega_0(v_i) + k  = 2k+2 \geq D_{\ell}(v_i) = 2k+2,$$ and vertices $v_i \in \{v_{2k+2},v_{2k+1},\dots,v_k\}$  are all cleaned during step $\ell+1$.   Further, we note that these vertices will each have $k+2$ brushes in the final configuration.  

Finally, we consider the vertices cleaned during phase 3: $\{v_{k-1},v_{k-2},\dots,v_0\}$.  No vertex $v_i \in \{v_{k-1},v_{k-2},\dots,v_0\}$ can be cleaned prior to step $\ell+2$ as $$\omega_\ell(v_i) = \omega_0(v_i)+k = i+k < D_\ell(v_i) = 2k+2.$$  However, all $v_i \in \{v_{k-1},v_{k-2},\dots,v_0\}$ are cleaned during step $\ell+2$ as $$\omega_{\ell+1}(v_i) = \omega_0(v_i) + 2k+3 = i+2k+3\geq D_{\ell+1}(v_i) = k-1.$$ 

The final configuration is
	$$\omega_{\ell+2}(v_i) =
	\begin{cases}
		i + 3 \cdot 2^{t^* - 1} - 3k - 5
			& \text{~for~} i = 3k - 2^{\ell-1} + 4, \dots, 3k + 2 \\
		i + 2^{\ell - 1} - 2k - 4
			& \text{~for~} i = 2k + 3, \dots, 3k - 2^{\ell - 1} + 3 \\
		k+2
			& \text{~for~} i = k, k + 1, \dots, 2k + 2 \\
		i + 2k + 3
			& \text{~for~} i = 0, 1, \dots k - 1
	\end{cases} $$
	
where $t^* = \lceil \log_2(3k-i+4)\rceil$ is the step at which $v_i$ was cleaned.  By a relabeling of vertices, configuration $\omega_{\ell+2}$ is equivalent to $\omega_0$.  We further note that at each step of the cleaning process, no vertex had more than $n-1 = 3k+2$ brushes.  \end{proof}

In the next lemma, we start with a $1$-clique configuration of $K_n$ and use it, along with the previous theorem, to build a $1$-clique configuration of $K_{3n+3}$.  In Theorem~\ref{thm:n=3kbetterbound}, vertices of $K_{3n+3}$ were cleaned in 3 phases, with $n$ vertices cleaned during phase 1. 

\begin{lemma}\label{lemma:+0} Let $n \in \mathbb{N}$.  There exists a $1$-clique configuration that cleans $K_{3n+3}$ using $2  S_{n} + 3n^2+8n+6$ brushes.\end{lemma}  

\begin{proof} Let $n \in \mathbb{N}$ and label the vertices of $K_{n}$ as $v_0,v_1,\dots,v_{n-1}$.  Let $\omega_0'$ be a $1$-clique configuration of $K_{n}$.   Label the vertices of $K_{3n+3}$ as $u_0,u_1,u_2,\dots,u_{3n+2}$, and set $$\omega_0(u_j) = \left\{ \begin{array}{ll}
	\omega_0'(v_j) & \text{if }    ~ 0 \leq j \leq n-1\\
	n+2 & \text{if }  ~ n\leq j \leq 2n+2 \\
	\omega_0'(v_{j-2n-3})+2n+3 & \text{if }  ~ 2n+3 \leq j \leq 3n+2.  \end{array} \right.$$  Let $A = \{u_j~|~2n+3 \leq j \leq 3n+2\},$ $B=\{u_j~|~n \leq j \leq 2n+2\},$ $C=\{u_j~|~0 \leq j \leq n-1\}.$  Sets $A$, $B$, and $C$ are cleaned during phases $1$, $2$, and $3$, respectively.   

We first observe that no vertex of $B \cup C$ can be cleaned until $n$ vertices of $K_{3n+3}$ have been cleaned.  If only $n-1$ vertices have been cleaned, then a vertex in $B \cup C$ will have at most $(n+2)+(n-1) = 2n+1$ brushes, but will have $2n+3$ dirty incident neighbours.  Thus, in phase $1$, only vertices of $A$ will be cleaned.\\

\noindent\underline{Phase 1:} Let $j \in \{ 2n+3,\dots, 3n+2\}$.  We aim to show that if $v_{j-2n-3}$ is cleaned during step $t$ in $K_n$ then $u_j$ is cleaned during step $t$ in $K_{3n+3}$.  Suppose $n>1$ (one can manually check that the configuration is a $1$-clique configuration for $n=1$).  Obviously, the previous statement holds for $t=1$ and suppose that the statement holds for all $t<t'$ for some step $t'$.  By induction, we prove the statement holds for $t=t'$.  In $K_n$, suppose vertex $v_{j-2n-3}$ is cleaned during step $t'$ and $x$ vertices were cleaned during earlier steps;  then \begin{equation}\label{ee}\omega_0'(v_{j-2n-3})+x \geq n-1-x ~\Rightarrow~ \omega_0'(v_{j-2n-3}) \geq n-1-2x.\end{equation}  

Using (\ref{ee}), we see that vertex $u_j$ in $K_{3n+3}$ is cleaned during step $t'$: $$\omega_{t'-1}(u_j) = \omega_0(u_j)+x = \omega_0'(v_{j-2n-3})+2n+3+x  \geq 3n+2-x  = D_{t'-1}(u_j).$$

Suppose the vertices of $A$ (and of $K_n$) are cleaned by step $\kappa$.  Since $\omega_0'$ is a $1$-clique configuration in $K_n$, there is a one-to-one correspondence between the sets $\{\omega_\kappa(u_{3n+2}),\omega_\kappa(u_{3n+1}),\dots, \omega_\kappa(u_{2n+3})\}$ and $\{\omega_0'(v_0),\omega_0'(v_1), \dots, \omega_0'(v_{n-1})\}.$  Further, since $\omega_{t-1}'(v_{j-2n-3}) \leq n-1$ (as $\omega_0'$ is a $1$-clique configuration) and $\omega_{t-1}'(v_{j-2n-3})=\omega_0'(v_{j-2n-3})+x$, we conclude $$\omega_{t-1}(u_j) = \omega_0'(v_{j-2n-3})+x+(2n+3) \leq (n-1)+(2n+3) = 3n+2.$$  Thus, a vertex in $A$ has at most $3n+2$ brushes at any step.\\

\noindent\underline{Phase 2:} Next, we observe that no vertex of $C$ can be cleaned at step $\kappa+1$: each vertex of $C$ has at most $(n-1)+n = 2n-1$ brushes, but $2n+2$ dirty incident neighbours (since $|A|=n$ and $|B|=n+3$).  Similarly for $u_j \in B$, $\omega_{\kappa-1}(u_j) \leq 2n+1 < D_{\kappa-1}(u_i)$ and $D_{\kappa-1}(u_j) \geq 2n+3$, so no vertex of set $B$ can be cleaned at step $\kappa$ (or earlier).

However, for each $u_j \in B$, $\omega_\kappa(u_j) = 2n+2$.  Thus, each vertex of $B$ is cleaned at step $\kappa+1$, leaving $\omega_{\kappa+1}(u_j) = n+2$ for each $v_j \in B$.  Clearly there is a one-to-one correspondence between the sets $\{\omega_{\kappa+1}(u_{n}),\omega_{\kappa+1}(u_{n+1}),\dots,\omega_{\kappa+1}(u_{2n+2})\}$ and $\{\omega_0(u_{n}),\omega_0(u_{n+1}),\dots,\omega_0(u_{2n+2})\}$.  Further, we note that a vertex in $B$ has at most $2n+2<3n+2$ brushes at any step.\\ 

\noindent\underline{Phase 3:} Finally, we consider the vertices of $C$.  For $u_j \in C$, $\omega_{\kappa+1}(u_j)=\omega_0'(v_j)+ 2n+3 \geq 2n+3$ and since $|C|=n$, every vertex of $C$ is cleaned at step $\kappa+2$.  Thus, there is a one-to-one correspondence between the sets $\{\omega_{\kappa+2}(u_0), \omega_{\kappa+2}(u_1),\dots, \omega_{\kappa+2}(u_{n-1}) \}$ and $\{\omega_0(u_{2n+3}), \omega_0(u_{2n+4}), \dots, \omega_0(u_{3n+2})\}$.  Further, we note that a vertex in $C$ has at most $3n+2$ brushes at any step.   \end{proof}

\begin{lemma}\label{lemma:+1} Let $n \in \mathbb{N}$.  There exists a $1$-clique configuration that cleans $K_{3n+1}$ using $2S_{n} + 3n^2+2n$ brushes where $S_n$ is a $1$-clique configuration for $K_n$.\end{lemma}

\begin{lemma}\label{lemma:+2} Let $n \in \mathbb{N}$.  There exists an initial $1$-clique configuration that cleans $K_{3n+2}$ using $2  S_{n} + 3n^2+5n+2$ brushes where $S_n$ is a $1$-clique configuration for $K_n$.
\end{lemma}

The proofs of Lemmas~\ref{lemma:+1} and~\ref{lemma:+2} are extremely similar to the proof of Lemma~\ref{lemma:+1} and consequently have been omitted.  We do, however, provide the initial configurations used to prove Lemmas~\ref{lemma:+1} and~\ref{lemma:+2}.  Label the vertices of $K_{n}$ as $v_0,v_1,\dots,v_{n-1}$. and let $\omega_0'$ be a $1$-clique configuration of $K_{n}$.   Label the vertices of $K_{3n+1}$ as $u_0,u_1,u_2,\dots,u_{3n}$ and set $$\omega_0(u_j) =
	\begin{cases}
		\omega_0'(v_j)
			& \text{~if~} 0 \leq j \leq n-1 \\
		n
			& \text{~if~} n \leq j \leq 2n \\
		\omega_0'(v_{j-2n-1})+2n+1
			& \text{~if~} 2n+1 \leq j \leq 3n.
	\end{cases}$$   Label the vertices of $K_{3n+2}$ as $u_0,u_1,u_2,\dots,u_{3n+1}$, and set $$\omega_0(u_j) = \left\{ \begin{array}{ll}
	\omega_0'(v_j) & \text{if }    ~ 0 \leq j \leq n-1\\
	n+1 & \text{if }  ~ n\leq j \leq 2n+1 \\
	\omega_0'(v_{j-2n-2})+2n+2 & \text{if }  ~ 2n+2 \leq j \leq 3n+1.  \end{array} \right.$$
	
Iteratively applying Lemmas~\ref{lemma:+0}-\ref{lemma:+2}, yields the next corollary.

\begin{corollary}\label{cor}
Let $n_0 \in \mathbb{N}$ and for $i \in \mathbb{N}$, let $n_i = 3n_{i-1}+d_i$ for $d_i \in \{1,2,3\}$.  Then there exists an initial $1$-clique configuration that cleans $K_{n_i}$ using $2 \cdot S_{n_{i-1}} + \frac{1}{3}n_i^2+O(n_i)$ brushes.
\end{corollary}

\begin{theorem}\label{theorem:b}  Let $n_0\in \mathbb{N}$ and for $i \in \mathbb{N}$, let $n_i = 3n_{i-1}+d_i$ for $d_i \in \{1,2,3\}$.  Then \begin{equation}\label{eqn}S_{n_i} \leq  \Big[ \frac{3}{7}+\frac{1}{63}\Big(\frac{2}{9}\Big)^{i}\Big] n_i^2 + O(n_i).\end{equation}
\end{theorem}

\begin{proof} Let $n_0 \in \mathbb{N}$ and for $i \in \mathbb{N}$, let $n_i = 3n_{i-1}+d_i$ for $d_i \in \{1,2,3\}$.  By Theorem~\ref{thm:n=3kbetterbound} along with Theorems 4.8~\cite{parallel} and 4.10~\cite{parallel}, we know there exists an initial $1$-clique configuration that cleans $K_{n_0}$ using $\frac{4}{9}n_0^2 + O(n_0)$ brushes:  so $S_{n_0} \leq \frac{4}{9}n_0^2+O(n_0)$.

Assume (\ref{eqn}) holds for all $i\leq k$ for $i, k\in \mathbb{Z}^+$. Now consider $i=k+1$ then from Corollary \ref{cor} and the inductive hypothesis it follows that: 
\begin{equation}
\begin{aligned}
S_{n_{k+1}} &\leq 2\cdot S_{n_k} + \frac{1}{3}n_{k+1}^2+O(n_k)\\
	  &\leq 2\Big[\Big( \frac{3}{7}+\frac{1}{63}\Big(\frac{2}{9}\Big)^k\Big) n_k^2 + O(n_k)\Big]+\frac{1}{3}n_{k+1}^2+O(n_{k+1}) \\
      &\leq \Big[ \frac{3}{7}+\frac{1}{63}\Big(\frac{2}{9}\Big)^{{k+1}}\Big] n_{k+1}^2 + O(n_{k+1})
\end{aligned}
\end{equation}
\end{proof}

Theorem~\ref{cor:end} follows immediately from Theorem~\ref{theorem:b}.

\begin{theorem}\label{cor:end}  Let $n_0$ be a non-negative integer and for $i \in \mathbb{Z}^+$, let $n_i = 3n_{i-1}+d_i$ for $d_i \in \{1,2,3\}$.  Then $$cpb(K_{n_i}) \leq  \Big[ \frac{3}{7}+\frac{1}{63}\Big(\frac{2}{9}\Big)^{i}\Big] n_i^2 + O(n_i).$$\end{theorem}

\begin{corollary}\label{theorem:lim} Let $n_0$ be a non-negative integer and for $i \in \mathbb{Z}^+$, let $n_i = 3n_{i-1}+d_i$ for $d_i \in \{1,2,3\}$. Then $$\lim_{i \to \infty}\ \frac{b(K_{n_i})}{cpb(K_{n_i})}\ge\frac{7}{12}.$$ \end{corollary} 

Corollary~\ref{theorem:lim} disproves the conjecture~(\ref{conj}) of~\cite{parallel}. \\



\begin{thebibliography}{x}

\bibitem{Zebra1} A.J. Benson, D. Raikow, J. Larson, A. Fusaro, and A.K. Bogdanoff. 2017. Dreissena polymorpha. USGS nonindigenous aquatic species database, Gainesville, FL. https://nas.er.usgs.gov/queries/factsheet.aspx?speciesid=5~Revision Date: 6/26/2014.

\bibitem{note} S. Gaspers, M.E. Messinger, R.J. Nowakowski, P. Pra{\l}at, Clean the graph before you draw it! {\it Information Processing Letters} 109 (2009) 463-467.

\bibitem{parallel} S. Gaspers, M.E. Messinger, R.J. Nowakowski, P. Pra{\l}at, Parallel cleaning of a network with brushes, Discrete Applied Mathematics 158 (2010) 467--478.

\bibitem{Zebra} B. Hobbs, J. Kahabka, Underwater cleaning techniques used for removal of zebra mussels at the FitzPatrick Nuclear Power Plant, Proceedings of the Fifth International Zebra Mussel and other Aquatic Nuisance Organisms Conference, Toronto, Canada (1995) 211-226.

\bibitem{Zebra2} S.R. Kotler, E.C. Mallen, K.M. Tammus, Robotic removal of zebra mussel accumulations in a nuclear power plant screenhouse, Proceedings of the Fifth International Zebra Mussel and other Aquatic Nuisance Organisms Conference, Toronto, Canada (1995).

\bibitem{Sable} S. McKiel, Graph cleaning, MSc Thesis, Dalhousie University (2007). 

\bibitem{MNP} M.E. Messinger, R.J. Nowakowski, P. Pra{\l}at, Cleaning a network with brushes, Theoretical Computer Science 399 (2008) 191--205.

\bibitem{greedy} M.E. Messinger, R.J. Nowakowski, P. Pra{\l}at, N.C. Wormald, Cleaning random d-regular graphs with brushes using a degree-greedy  algorithm, {\it Proceedings of the 4th Workshop on Combinatorial and Algorithmic Aspects of Networking (CAAN2007)}, Lecture Notes in  Computer Science, Springer (2007) 13-26.  

\end{thebibliography}
\end{document}